\documentclass[a4paper,11pt]{amsart}
\usepackage{amssymb,amscd,amsmath}
\usepackage[dvips]{color}

\theoremstyle{plain}
\newtheorem{thm}{Theorem}[section]

\newtheorem{lem}[thm]{Lemma}
\newtheorem{cor}[thm]{Corollary}

\theoremstyle{definition}

\theoremstyle{remark}
\newtheorem{rem}{Remark}[section]
\newtheorem{ex}[rem]{Example}

\newcommand{\forme}[1]{}

\begin{document}
\title{Characterization of finite metric spaces by their isometric sequences}

\author{Mitsugu Hirasaka}
\author{Masashi Shinohara}

\address{Department of Mathematics, College of Sciences, Pusan National University,
63 Beon-gil 2, Busandaehag-ro, Geumjung-gu, Busan 609-735, Korea.}

\address{Department of Education, Faculty of Education, Shiga University, 2-5-1 Hiratsu, Otsu,  SHIGA 520-0862, Japan}
\email{hirasaka@pusan.ac.kr}
\email{shino@edu.shiga-u.ac.jp}
\email{}

\date{\today}

\thanks{}
\thanks{}
\thanks{}

\maketitle


\begin{abstract}
Let $(X,d)$ be a finite metric space with $|X|=n$.
For a positive integer $k$ we define $A_k(X)$ to be the quotient set of all $k$-subsets of $X$
by isometry, and we denote $|A_k(X)|$ by $a_k$.
The sequence $(a_1,a_2,\ldots,a_{n})$ is called the \textit{isometric} sequence of $(X,d)$.
In this article we aim to characterize finite metric spaces by their isometric sequences
under one of the following assumptions: (i) $a_k=1$ for some $k$ with $2\leq k\leq n-2$;
(ii) $a_k=2$ for some $k$ with $4\leq k\leq \frac{1+\sqrt{1+4n}}{2}$;
(iii) $a_3=2$; (iv) $a_2=a_3=3$. Furthermore, we give some criterion on how to embed such finite metric spaces to Euclidean spaces.  
We give some maximum cardinalities of subsets in the $d$-dimensional Euclidean space with small $a_3$, which are analogue problems on a sets with few distinct triangles discussed by 
Epstein, Lott, Miller and Palsson. 
\end{abstract}

\section{Introduction}
In the study of distance sets (see \cite{LRS} and \cite{BBS}) we deal with 
a finite set $S$ in a Euclidean space and aim to characterize $S$ by  
the number of possible distances among the elements of $S$. In this article we put our focus on
not only distances but also triangles or other kinds of subsets. Furthermore, we deal with 
a finite set in not only Euclidean spaces but also general metric spaces. Now we shall remind of
terminologies on metric spaces and their isometry. 

Let $(X,d)$ be a metric space where $d:X\times X\to \mathbb{R}_{\geq 0}$ is a
metric function. For $A, B\subseteq X$ we say that $A$ is \textit{isometric} to $B$ if
there exists a bijection $f:A\to B$ such that $d(x,y)=d(f(x),f(y))$ for all $x,y\in A$.
For a positive integer $k$ we denote the family of $k$-subsets of $X$ by $\binom{X}{k}$, i.e.,
\[\binom{X}{k}=\{Y\subseteq X\mid |Y|=k\},\]
and we define $A_k(X)$ to be the quotient set of $\binom{X}{k}$ by the isometry derived from $d$, i.e.,
\[A_k(X)=\left\{[Y]\mid Y\in \binom{X}{k}\right\}\]
where $[Y]=\left\{Z\in \binom{X}{k}\mid \mbox{$Y$ is isometric to $Z$}\right\}$.
If $X$ is a finite set, then we define the \textit{isometric} sequence of $(X,d)$ to be
\[(a_1,a_2,\ldots, a_n)\]
where $a_i=|A_i(X)|$ and $|X|=n$.
Clearly, $a_1=1$ and $a_n=1$, and $a_2=1$ implies $a_i=1$ for all $i$ with $1\leq i\leq n$.
The following are examples of isometric sequences of finite sets in Euclidean spaces:
\begin{enumerate}
\item The four vertices of a square: $(1,2,1,1)$; 
\item The five vertices of a regular pentagon: $(1,2,2,1,1)$; 
\item The six vertices of a octahedron: $(1,2,2,2,1,1)$;
\end{enumerate}

Notice that, for $x,y,z,w\in X$,
\[
\mbox{$[\{x,y\}]=[\{z,w\}]$ if and only if $d(x,y)=d(z,w)$.}\]
Thus, we may identify the elements of $A_2(X)$ as the elements of 
\[\{d(x,y)\mid x,y\in X,x\ne y\}.\]

Recall that any connected graph is a metric space with its graph distance.
The following are examples of isometric sequences of finite connected graphs:
\begin{enumerate}
\item The complete graph $K_n$: $(1,1,\ldots,1)$;
\item The complete bipartite graph $K_{3,3}$: $(1,2,2,2,1,1)$;
\item The cycle $C_6$: $(1,3,3,3,1,1)$.
\end{enumerate}
In order to generate isometric sequences we need the partition $\{E_\alpha\}_{\alpha\in A_2(X)}$ of $\binom{X}{2}$
where \[E_\alpha=\{\{x,y\}\mid d(x,y)=\alpha\}\]
but not the exact values of $\alpha\in A_2(X)$.
Actually, the discrete partition of $\binom{X}{2}$ 
induces the isometric sequence $(1,\binom{n}{2},\binom{n}{3},\ldots,\binom{n}{n-1},1)$.

In this article we aim to characterize finite metric spaces by their isometric sequences.
For example, starting from the isometric sequence $(1,2,1,1)$ we obtain 
the partition of $\binom{X}{2}$ into $E_\alpha$ and $E_\beta$ where
$(X,E_\alpha)$ is a cycle of length $4$ and $(X,E_\beta)$ is its complement. 
This is the first step in determining the partition of $\binom{X}{2}$ into $a_2$ cells
up to permutations of $X$.

As the second step we consider how to embed given finite metric spaces into Euclidean spaces
or other metric spaces.
In \cite{Ne} Neumair gives a criterion on embeddings into Euclidean spaces as follows:
Let $D$ denote the distance matrix of a finite metric space $(X,d)$, i.e., the rows and columns of $D$ are indexed by the elements of $X$ and
the $(x,y)$-entry of $D$ is defined to be $d(x,y)^2$. We set a matrix $G$ to be
\begin{equation}\label{eq:G}
 G=-(I-1/nJ)D(I-1/nJ)
\end{equation}
where $I$ and $J$ are the identity matrix and 
the all one matrix of degree $n=|X|$, respectively. 
Then $X$ is embedded into $\mathbb{R}^d$ where $d=\mathrm{rank}(G)$ if and only if $G$ is positive semidefinite.

For example, each of such matrices with the isometric sequence $(1,2,1,1)$ forms
\[\begin{pmatrix}
0 & a & b& b \\
a & 0 & b & b \\
b & b & 0 & a\\
b & b & a & 0
\end{pmatrix}\]
for some $a,b\in \mathbb{R}_{>0}$ with $a\ne b$ and a suitable ordering of the elements of $X$.
The characteristic polynomial $\det(tI-G)$ of $G$ is computed
as 
\[t(t-a)^2(t-2b+a).\]
Therefore, it is embedded into a Euclidean space if and only if $a\leq 2b$,
and the equality holds if and only if it can be embedded into $\mathbb{R}^2$.

In such a way we are required to find exact values of $\alpha\in A_2(X)$ such that
$G$ is positive definite and $\mathrm{rank}(G)$ is minimal.

In Section~2 we prepare some terminologies and basic results, and in Section~3 we discuss finite metrics spaces with the following isometric sequences $(a_1,a_2,\ldots,a_n)$:
\begin{enumerate}
\item $a_k=1$ for some $k$ with $2\leq k\leq n-2$;
\item $a_k=2$ for some $i$ with $4\leq k\leq \frac{1+\sqrt{1+4n}}{2}$;
\item $a_2=a_3=2$;
\item $a_2=a_3=3$.
\end{enumerate}
In Section~4, we discuss how to embed the above finite metrics spaces into Euclidean spaces 
and give maximum cardinalities of subsets in the $d$-dimensional Euclidean space 
with small $a_3$,   
which are analogue problems 
on a sets with few distinct triangles discussed by Epstein, Lott, Miller and Palsson 
\cite{ELMP}. 

\section{Preliminaries}
Throughout this article we assume that $(X,d)$ is a finite metric space with its isometric sequence
$(a_1,a_2,\ldots, a_n)$.
For short we shall write $x_1x_2\cdots x_k$ instead of $[\{x_1,x_2,\ldots, x_k\}]$ for $\{x_1,x_2,\ldots,x_k\}\in \binom{X}{k}$, so that
\[x_1x_2\cdots x_k=x_{\sigma(1)}x_{\sigma(2)}\cdots x_{\sigma(k)}\]
for each $\sigma\in S_k$ where $S_k$ is the symmetric group of degree $k$.
Recall that $x_1x_2\cdots x_k$ can be identified with the distance matrix $(d(x_i,x_j)^2)_{1\leq i,j\leq k}$.
Since $(d(x_i,x_j)^2)_{1\leq i,j\leq k}$ is a symmetric matrix whose diagonal entries are zero,
we may identify
\[\mbox{$x_1x_2$ with $d(x_1,x_2)$, and $x_1x_2x_3$ with
$(d(x_1,x_2),d(x_2,x_3),d(x_3,x_1))$,}\]
 which we often express without the commas or
the parenthesis, like
\[\mbox{$\alpha\beta\gamma$ for $x_1x_2x_3$}\]
 with
$\alpha=d(x_1,x_2)$, $\beta=d(x_2,x_3)$ and $\gamma=d(x_3,x_1)$.

For $\alpha\in A_2(X)$ we define a binary relation $R_\alpha$ on $X$ to be
\[R_\alpha=\{(x,y)\in X\times X\mid d(x,y)=\alpha\},\]
so that $\{R_\alpha\}_{\alpha\in A_2(X)}$ is a partition of
$\{(x,y)\in X\times X\mid x\ne y\}$.
Also, we define $E_\alpha$ to be
\[E_\alpha=\left\{\{x,y\}\in \binom{X}{2}\mid (x,y)\in R_\alpha\right\},\]
so that $(X,E_\alpha)$ is a simple graph for each $\alpha\in A_2(X)$.
For a binary relation $R$ on $X$ and $x\in X$ we set
\[R(x)=\{y\in X\mid (x,y)\in R\}.\]
For a finite metric space $(X_1,d_1)$ we say that $(X,d)$ is \textit{isomorphic} to $(X_1,d_1)$
if there exist bijections $f:X\to X_1$ and $g:A_2(X)\to A_2(X_1)$
such that, for all $x,y\in X$,
\[g(d(x,y))\Leftrightarrow d_1(f(x),f(y)).\]
Remark that any two isometric metric spaces are isomorphic, but the converse does not hold in general.

For $S,T\subseteq X$ we define a vector $v(S,T)$ whose entries are indexed by the elements of $A_2(X)$ as follows:
\[v(S,T)_\alpha=|(S\times T)\cap R_\alpha|\]
where we shall write $v(x,T)$ and $v(S,y)$ instead of $v(\{x\},T)$ and $v(S,\{y\})$
when $x,y\in X$.
\begin{lem}\label{lem:ST}
For all $S,T,U\subseteq X$ we have the following:
\begin{enumerate}
\item $v(S,T)=v(T,S)$;
\item If $S\cap T=\emptyset$, then $v(S\cup T,U)=v(S,U)+v(T,U)$;
\item If $S$ is isometric to $T$, then $v(S,S)=v(T,T)$.
\end{enumerate}
\end{lem}
\begin{proof}
Since $R_\alpha$ is symmetric, (i) follows.

Since $(S\cup T)\times U=(S\times U)\cup (T\times U)$ and $(S\times U)\cap (T\times U)=\emptyset$ if $S\cap T=\emptyset$, we have
\[[(S\cup T)\times U]\cap R_\alpha=[(S\times U)\cap R_\alpha]\cup [(T\times U)\cap R_\alpha],\]
and hence (ii) follows.

(iii) follows from the definition of isometry.
\end{proof}

\begin{lem}\label{lem:SS}
We have the following:
\begin{enumerate}
\item $|\{v(S,S)\mid S\in \binom{X}{k}\}|\leq a_k$;
\item For each $S\in \binom{X}{k-1}$ we have $|\{v(x,S)\mid x\in X\setminus S\}|\leq a_{k}$;
\item For each $\alpha\in A_2(X)$ and $\emptyset\ne S\subseteq X$ there exists $y\in S$ such that
\[v(X\setminus S, y)_\alpha\geq v(X\setminus S,S)_\alpha/|S|.\]
\end{enumerate}
\end{lem}
\begin{proof}
By the contraposition of Lemma~\ref{lem:ST}(iii), if $v(S,S)\ne v(T,T)$,
then $[S]\ne [T]$, and hence (i) follows.

By Lemma~\ref{lem:ST}(i),(ii),
\[v(\{x\}\cup S,\{x\}\cup S)=v(x,x)+v(x,S)+v(S,x)+v(S,S)=2v(x,S)+v(S,S).\]
Since $\{v(\{x\}\cup S,\{x\}\cup S)\mid x\in X\setminus S \}\subseteq \{v(T,T)\mid T\in \binom{X}{k}\}$,
(ii) follows from (i).

(iii) follows from the pigeon hole principle with $v(X\setminus S,S)_\alpha$ pigeons
into $|S|$ halls.
\end{proof}

For a positive integer $k$  with $1\leq k\leq n$ we set
\[M_k=\{\alpha\in A_2(X)\mid \mbox{$v(X\setminus S,S)_\alpha\geq k|S|$ for
some $\emptyset \ne S\subseteq X$}\}.\]
\begin{lem}\label{lem:mk}
We have $M_k=\{\alpha\in A_2(X)\mid \exists x\in X;v(x,X)_\alpha\geq k\}$.
In particular, $\emptyset=M_n\subseteq M_{n-1}\subseteq \cdots\subseteq M_2\subseteq M_1=A_2(X)$.
\end{lem}
\begin{proof}
Let $\alpha\in M_k$. Then, by Lemma~\ref{lem:SS}(iii), there exists $y\in X$ such that
$v(y,X)_\alpha\geq k$. Conversely, if $v(x,X)_\alpha\geq k$,
then
\[v(X\setminus\{x\},\{x\})_\alpha\geq k,\]
and hence $\alpha\in M_k$.
\end{proof}

\begin{rem}
For each $\alpha\in A_2(X)$, $\alpha\notin M_2$ if and only if $(X,E_\alpha)$ is a matching on $X$,
and $\alpha\alpha\alpha\notin A_3(X)$ if and only if $(X,E_\alpha)$ is triangle-free.
\end{rem}

\begin{lem}\label{lem:major}
For each integer $k$ with $2\leq k\leq n$ we have $|M_{k-1}|\leq  a_k$.
\end{lem}
\begin{proof}
Let $\alpha,\beta\in M_{k-1}$ with $\alpha\ne \beta$.
Then, by Lemma~\ref{lem:mk},
each of $(X,E_\alpha)$ and $(X,E_\beta)$ contains a subgraph isomorphic to
a star with $k-1$ leaves.

We claim that the subsets spanned by the stars isomorphic to $K_{1,k-1}$ in $(X,E_\alpha)$ and $(X,E_\beta)$
are non-isometric.
Suppose the contrary, i.e., $f$ is an isometry from the previous one to the latter one.
Then $f$ moves the vertex of degree $k-1$ to an isolated vertex in $(X,E_\beta)$ since $\alpha\ne\beta$,
which contradicts that $K_{1,k-1}$ is connected.

Therefore, the above claim shows that the elements of $M_{k-1}$ can be embedded into $A_k(X)$.
\end{proof}

\forme{
\begin{lem}\label{lem:empty}
If $M_{k-1}=\emptyset$, then $(k-2)a_2\geq n-1$.
\end{lem}
\begin{proof}
Let $x\in X$. Then $X\setminus\{x\}=\bigcup_{\alpha\in A_2(X)}R_\alpha(x)$.
Since $M_{k-1}=\emptyset$, it follows from Lemma~\ref{lem:mk} that
\[n-1=\sum_{\alpha\in A_2(X)}|R_\alpha(x)|\leq a_2(k-2).\]
\end{proof}
}

\forme{
\begin{lem}\label{lem:mk2}
For each integer $k$ with $2\leq k\leq n$, if $a_k(k-1)^2+k-1\leq n$, then $1\leq |M_{k-1}|$.
\end{lem}
\begin{proof}
Let $S\in \binom{X}{k-1}$.
Then, by Lemma~\ref{lem:SS}(ii),
\[|\{v(S,x)\mid x\in X\setminus S\}|\leq a_k.\]
By the pigeon hall principle, one vector is repeated $\lfloor\frac{n-(k-1)}{a_k}\rfloor$ times
in $\{v(S,x)\mid x\in X\setminus S\}$.
By the assumption,
\[\frac{n-k+1}{a_k}\geq (k-1)^2\geq |S|(k-1).\]
Therefore, the elements of $A_2(X)$ corresponding to the nonzero entries of the repeated vectors are contained in $M_{k-1}$.
\end{proof}
}

\begin{lem}\label{lem:bound}
Let $\gamma\in A_2(X)\setminus M_{k-1}$ and $S\in\binom{X}{k}$ such that
the induced subgraph of $(X,E_\gamma)$ by $S$ contains a spanning forest.
If $k^2-k\leq n$, then the number of edges in the forest is at most $a_k-1$.
\end{lem}
\begin{proof}
Suppose that the spanning forest has exactly $c$ connected components.
Then $k-c$ equals the number of edges in the spanning forest.
Since every non-null forest has a leaf,
there exist $y_1,y_2,\ldots, y_{k-c}\in S$ such that
$y_i$ is a leaf in the subgraph of $(X,E_\gamma)$ induced by $S\setminus\{y_1,\ldots, y_{i-1}\}$ for each $i$ with $1\leq i\leq k-c$.

We claim that, for each $y\in S$,
there exists $x\in X\setminus S$ such that 
\[v(x,S\setminus\{y\})_\gamma=0.\]
Suppose the contrary, i.e., there exist $y\in S$ such
that $v(x,S\setminus\{y\})_\gamma>0$ for each $x\in X\setminus S$.
By the assumption of $k^2-k\leq n$, we have
\[v(X\cup \{y\}\setminus S,S\setminus\{y\})_\gamma\geq n-k+1\geq (k-1)^2
=|S\setminus\{y\}|(k-1).\]
By Lemma~\ref{lem:SS}(iii), there exists $z\in S\setminus \{y\}$ such that $v(z,X\setminus S)_\gamma\geq k-1$.
It follows from Lemma~\ref{lem:mk} that $\gamma\in M_{k-1}$, a contradiction to the assumption.

Applying the clam inductively for $y_i$ with $Y_i$ we obtain that $x_i\in X$ such that $v(x_i,Y_i\setminus\{y_i\})_\gamma=0$ where
\[Y_0=S, Y_1:=Y_0\cup\{x_1\}\setminus \{y_1\},\ldots, Y_i:=Y_{i-1}\cup \{x_i\}\setminus\{y_i\},\ldots,\]
so that
\[v(Y_0,Y_0)_\gamma >v(Y_1,Y_1)_\gamma>\cdots >v(Y_{k-c},Y_{k-c})_\gamma=0.\]
By Lemma~\ref{lem:SS}(i), we have $k-c+1\leq a_k$.
\end{proof}

\begin{lem}\label{lem:max}
For $S\in \binom{X}{k}$ and $\alpha\in A_2(X)$,
each permutation of $S$ which fixes each element in $\{x\in S\mid v(x,S)_\alpha<k-1\}$ is an isometry.
\end{lem}
\begin{proof}
It is a routine to check that such a permutation satisfies the definition of isometries.
\end{proof}

For $\emptyset\ne \Gamma\subseteq A_2(X)$ we say that $\Gamma$ is \textit{closed}
if, for all $\alpha,\beta\in \Gamma$ and $\gamma\in A_2(X)$,
$\alpha\beta\gamma\in A_3(X)$ implies $\gamma\in  \Gamma$.

\begin{lem}\label{lem:closed}
For $\emptyset\ne \Gamma\subseteq A_2(X)$, $\Gamma$ is closed if and only if
$(X,\bigcup_{\gamma\in \Gamma}E_\gamma)$ is a disjoint union of cliques.
\end{lem}
\begin{proof}
Set $R_0=\{(x,x)\mid x\in X\}$.
Then $\Gamma$ is closed if and only if $\bigcup_{\alpha\in \Gamma\cup\{0\}} R_\alpha$ is an equivalence relation on $X$.
Since the equivalence classes correspond to the connected components of
the graph $(X,\bigcup_{\gamma\in \Gamma}E_\gamma)$, the lemma holds.
\end{proof}

\section{What are obtained from isometric sequences}
\begin{thm}\label{thm:a1}
If $a_k=1$ for some $k$ with $2\leq k\leq n-2$, then $a_2=1$.
\end{thm}
\begin{proof}
Suppose the contrary, i.e., $a_2>1$.
Then there exist distinct $x,y,z\in X$ such that $d(x,y)\ne d(y,z)$.

We claim that $d(x,w)=\alpha$ and $d(w,z)=\beta$ for all $w\in X\setminus\{x,z\}$
where $\alpha:=d(x,y)$ and $\beta:=d(y,z)$.
Let $w\in X\setminus\{x,y,z\}$ and $S\in\binom{X\setminus\{x,y,z,w\}}{k-2}$.
We set
\[S_1:=S\cup\{x,y\}, S_2:=S\cup \{y,z\}, S_3:=S\cup\{x,w\}, S_4:=S\cup\{z,w\},\]
so that $S_j$ are isometric for each other since $a_k=1$ and $|S_j|=k$.
Counting $v(S_j,S_j)_\alpha$ with $j=1,2,3,4$
we obtain, letting $r(u):=v(u,S)_\alpha$ for $u\in \{x,y,z,w\}$, from Lemma~\ref{lem:ST}(ii) that
\[r(x)+r(y)+1=r(y)+r(z)=r(x)+r(w)+v(x,w)_\alpha=r(z)+r(w)+v(z,w)_\alpha.\]
Therefore,
\[r(x)+1=r(z), r(x)+v(x,w)_\alpha=r(z)+v(z,w)_\alpha,\]
and hence,
\[1\leq v(z,w)_\alpha+1=v(x,w)_\alpha\leq 1.\]
This implies $\alpha=xw$, and we obtain from the similar argument that
$zw=\beta$ by symmetry of $\alpha$ and $\beta$.

Applying the claim for an arbitrarily taken $w$
 we have $\alpha,\beta\in M_{k-1}$. But, by Lemma~\ref{lem:major}, $|M_{k-1}|\leq a_k=1$, a contradiction.
\end{proof}

\begin{thm}\label{thm:a2}
If $a_k=2$ for some $k$ with $4\leq k\leq \frac{1+\sqrt{1+4n}}{2}$, then
$a_2=2$ and one of the following holds:
\begin{enumerate}
\item $(X,E_\alpha)\simeq K_n\setminus K_2$ for some $\alpha\in A_2(X)$;
\item $(X,E_\alpha)\simeq K_{1,n-1}$
for some $\alpha\in A_2(X)$;
\end{enumerate}
\end{thm}
\begin{proof}
Let $\gamma,\delta\in A_2(X)\setminus M_{k-1}$ and $S,T\in \binom{X}{k}$ such that $v(S,S)_\gamma$ is maximal
and $v(T,T)_\delta$ is maximal.
Applying Lemma~\ref{lem:bound} we obtain that
\[\mbox{$|\binom{S}{2}\cap E_\gamma|=1$ and $|\binom{T}{2}\cap E_\delta|=1$.}\]
Since $k\geq 4$, it follows from the maximality of $v(S,S)_\gamma$ and $v(S,S)_\delta$ that
\[\mbox{$|E_\gamma|=1$ and $|E_\delta|=1$.}\]
We claim that $\gamma=\delta$. Otherwise, there exist $U, V,W\in \binom{X}{k}$ such that
\[\mbox{$v(U,U)_\gamma=v(U,U)_\delta=1$, $v(V,V)_\gamma=v(V,V)_\delta=0$ and $v(W,W)_\gamma<v(W,W)_\delta$.}\]
By Lemma~\ref{lem:SS}(i), $3\leq a_k$, a contradiction.
Thus, we conclude from the claim that $|A_2(X)\setminus M_{k-1}|\leq 1$.

Suppose $A_2(X)\setminus M_{k-1}=\{\gamma\}$.
Let $(x,y)\in R_\gamma$.

We claim that $|A_2(Y)|=1$ for each $Y\subseteq X$
unless $x,y\in Y$.
Otherwise, there exist $S,T\in \binom{Y}{k}$ such
that $S$ is not isometric to $T$ by Theorem~\ref{thm:a1}.
Then neither of $S$ nor $T$ is isometric to each element $U\in \binom{X}{k}$ with $x,y\in U$
since $v(U,U)_\gamma=1>v(S,S)_\gamma=v(T,T)_\gamma=0$, which contradicts $a_k=2$.
By the claim and Theorem~\ref{thm:a1}, $A_2(X\setminus\{x\})=A_2(X\setminus\{y\})=\{\alpha\}$ for some $\alpha\in A_2(X)$,
which implies $(X,E_\alpha)\simeq K_n\setminus K_2$.

Suppose $A_2(X)\setminus M_{k-1}=\emptyset$, i.e., $A_2(X)=M_{k-1}$.
Since $1<a_2$ and $|M_{k-1}|\leq a_k=2$ by Lemma~\ref{lem:major},
it follows that $A_2(X)=\{\alpha,\beta\}$ where $\alpha\ne \beta$.
By the assumption, we have $n\ge k(k-1)\ge 2(k+1)$. 
Since $a_2=2$ and $n\ge 2(k+1)$, we may assume that $\alpha\in M_{k+1}$ without loss of generality.
Therefore, there exists $x\in X$ such that $v(x,X)_\alpha \geq k+1$.

We claim that all elements in $\binom{R_\alpha(x)}{k-1}$ are isometric.
Let $A,B\in \binom{R_\alpha(x)}{k-1}$.
Then $A\cup\{x\}, B\cup \{x\}\in \binom{X}{k}$, and they are isometric
by the same argument as in the proof of Lemma~\ref{lem:major} with $a_k=2$.
So, there exists a bijection $f:A\cup\{x\}\to B\cup \{x\}$ which preserve the metric function.
By Lemma~\ref{lem:max}, there exists an isometry $g$ on $B\cup\{x\}$ which moves $f(x)$ to $x$. 
Since the composite $gf$ is also an isometry which fixed $x$,
$A$ is isometric to $B$.

By the claim, $|A_{k-1}(R_\alpha(x))|=1$. Applying Theorem~\ref{thm:a1} for the sub-metric space $R_\alpha(x)$
we obtain $|A_{2}(R_\alpha(x))|=1$, i.e.,
\[\mbox{$A_{2}(R_\alpha(x))=\{\alpha\}$ or $A_{2}(R_\alpha(x))=\{\beta\}$.}\]
Thus, a clique of size $k$ exists in $(X,E_\alpha)$ or $(X,E_\beta)$.
Without loss of generality we may assume that $(X,E_\alpha)$ contains a clique of size $k$,
so that $S$ induces a clique of size $k$ whenever $S\in \binom{X}{k}$ induces a connected subgraph of $(X,E_\alpha)$.

Let $Y$ be a clique of maximal size in $(X,E_\alpha)$.
If $(y,z)\in R_\alpha$ for some $y\in Y$ and $z\in X\setminus Y$ and $T\in \binom{Y}{k-1}$ with $y\in T$,
then $T\cup \{z\}$ induces a clique in $(X,E_\alpha)$ by what we obtained in the last paragraph.
Since $T$ is arbitrarily taken, it follows that $Y\cup\{z\}$ is a clique in $(X,E_\alpha)$, which contradicts
the maximality of $|Y|$.
Therefore, every ordered pair in $Y\times(X\setminus Y)$ is belonged to $R_\beta$.

Finally, we show that $|X\setminus Y|=1$, which implies that $(X,E_\beta)\simeq K_{1,n-1}$.
Suppose the contrary, i.e., $x,z\in X\setminus Y$ with $x\ne z$.
Let $Y_0\in  \binom{Y}{k}$, $Y_1\in\binom{Y}{k-1}$ and $Y_2\in \binom{Y}{k-2}$.
Then
\[\mbox{$Y_0$, $Y_1\cup\{x\}$, $Y_2\cup\{x,z\}\in \binom{X}{k}$}\]
 such that they are mutually non-isometric
since
\[v(Y_0,Y_0)_\alpha>v(Y_1\cup\{x\},Y_1\cup\{x\})_\alpha>v(Y_2\cup \{x,z\},Y_2\cup\{x,z\})_\alpha.\]
This is a contradiction to $a_k=2$ by Lemma~\ref{lem:SS}(i).
\end{proof}

For the remainder of this section we assume that 
\[a_3\leq 3, \,  5\leq n.\]

\begin{lem}\label{lem:xyz}

For all $x,y,z\in X$ with $|\{xy,yz,zx\}|=3$ and each $u\in X\setminus \{x,y,z\}$
we have the following:
\begin{enumerate}
\item At least one of $\{uxy,uyz,yzx\}$ equals $xyz$;
\item $\{ux,uy,uz\}\subseteq \{xy,yz,zx\}$;
\item $A_2(X)=\{xy,yz,zx\}$;
\item $a_3=3$.
\end{enumerate}
\end{lem}
\begin{proof}
Let $x,y,z\in X$ such that
\[\mbox{$\alpha:=d(x,y)$, $\beta:=d(y,z)$ and $\gamma:=d(z,x)$ are distinct.}\]

(i) Suppose that (i) does not hold, i.e.,
\begin{equation}\label{eq:100}
xyz \notin \{uxy, uyz,uzx\}.
\end{equation}
Then $uxy,uyz,uzx\in A_3(X)\setminus\{xyz\}$. Since $a_3\leq 3$, at least two of them equal.
By symmetry we may assume that $uxy=uyz$ without loss of generality.
Since $\alpha,\beta \in A_2(\{u,x,y\})=A_2(\{u,y,z\})$ and $\alpha\ne \beta$,
we have $\beta\in \{ux,uy\}$.

If $\beta=ux$, then $uz \ne \alpha$ since $xyz\ne xuz$ by (\ref{eq:100}),
and hence $uy=\alpha$ since $uxy=uyz$.
This implies that $uxy=\alpha\alpha\beta$, and hence, $uz=\alpha$, a contradiction.

If $\beta=uy$, then $uz=\alpha$ since $uxy=uyz$ and $xy=\alpha$, and hence $ux=\beta$ since $uxy=uyz=\alpha\beta\beta$.
This implies that $uxz=\alpha\beta\gamma=xyz$, a contradiction to (\ref{eq:100}).

(ii) By (i), $xyz\in  \{uxy, uyz,uzx\}$. By symmetry, it suffices to show (ii) under the assumption of $xyz=uxy$.
Then
\begin{equation}\label{eq:15}
(ux,uy)\in \{(\gamma,\beta),(\beta,\gamma)\}.
\end{equation}
Let $v\in X\setminus\{x,y,z,u\}$.
For short we put
\[\mbox{$\delta:=uz$ and $uv=\epsilon$.}\]
Now we will prove $\delta\in\{\alpha,\beta,\gamma\}$.
Suppose not, i.e.,
\begin{equation}\label{eq:20}
\delta\notin\{\alpha,\beta,\gamma\}. 
\end{equation}
If $(ux,uy)=(\gamma,\beta)$,
then $A_3(X)=\{\alpha\beta\gamma,\gamma\gamma\delta,\beta\beta\delta\}$ since $a_3\leq 3$ and $|\{\alpha,\beta,\gamma,\delta\}|=4$.
Since $\{xyz\}=\{[Y]\in A_3(X)\mid \alpha\in A_2(Y)\}$,
$(vx,vy)\in \{(\gamma,\beta),(\beta,\gamma)\}$.
In the former case we have $\delta=\epsilon$ and
in the latter case we have $uv=\alpha$, each of the cases induces a contradiction that
\[uvz\notin \{\alpha\beta\gamma,\gamma\gamma\delta,\beta\beta\delta\}.\]
Therefore,
we have
\[(ux,uy)=(\beta,\gamma). \]
Applying (i) for $xyz$ with $v$ we obtain
\[xyz\in  \{vxy, vyz,vzx\}.\]

We will show a contradiction for each case of $xyz\in  \{vxy, vyz,vzx\}$ in the following three paragraphs:

If $xyz=vxy$, then
\[(vx,vy)\in \{(\gamma,\beta),(\beta,\gamma)\}.\]
In the former case, putting $\mu:=vz$
\[\mbox{$vxz=\gamma\gamma\mu$, $vyz=\beta\beta\mu$, $xyz=\alpha\beta\gamma$ and $uxz=\beta\gamma\delta$}\]
are distinct elements of $A_3(X)$, a contradiction to $a_3\leq 3$.
In the latter case
\[\mbox{$uvy=\gamma\gamma\epsilon$, $uvx=\beta\beta\epsilon$, $xyz=\alpha\beta\gamma$ and $uxz=\beta\gamma\delta$}\]
are distinct elements of $A_3(X)$ by (\ref{eq:20}), a contradiction to $a_3\leq 3$.

If $xyz=vyz$,
then
\[(vz,vy)\in \{(\gamma,\alpha),(\alpha,\gamma)\}.\]
In the former case, putting $\nu:=d(v,x)$
\[\mbox{$vxz=\gamma\gamma\nu$, $vxy=\alpha\alpha\nu$, $xyz=\alpha\beta\gamma$ and $uxz=\beta\gamma\delta$}\]
are distinct elements of $A_3(X)$, a contradiction to $a_3\leq 3$.
In the latter case
\[\mbox{$uvy=\gamma\gamma\epsilon$, $uvz=\alpha\delta\epsilon$, $xyz=\alpha\beta\gamma$ and $uxz=\beta\gamma\delta$}\]
are distinct elements of $A_3(X)$, a contradiction to $|A_3(X)|\leq 3$.

If $xyz=vzx$,
then, by symmetry of $\beta$ and $\gamma$, we can get a contradiction similar to the case of $xyz=vyz$.

Therefore, (ii) holds.

(iii) Let $u,v\in X$ with $\epsilon:=uv\notin \{\alpha,\beta,\gamma\}$.
Then, by (ii), $\{u,v\}\cap \{x,y,z\}=\emptyset$.
By (i), $xyz\in \{uxy,uyz,yzx\}$.
Applying (ii) for $uxy$, $uyz$ and $uzx$ with $v\in X$ we obtain
$uv\in \{\alpha,\beta,\gamma\}$, a contradiction.
Therefore, $\epsilon \in \{\alpha,\beta,\gamma\}$.
Since $u,v$ are arbitrarily taken, (iii) holds.

(iv) Since we assume that $a_3\leq 3$ and $a_2\geq 3$, it follows from Theorem~\ref{thm:a1} that
$a_3\in\{2,3\}$. Suppose $a_3=2$.
Then, by (i), we may assume that $xyz=uxy$.
Then $(ux,uy)\in \{(\beta,\gamma), (\gamma,\beta)\}$.
But, the latter case does not occur, otherwise $uyz,uxz,xyz$ are distinct elements in $A_3(X)$, a contradiction to $a_2=2$.
Therefore, $ux=\beta$ and $uy=\gamma$.

On the other hand, by (i), $xyz\in \{vxy,vyz,vzx\}$.
But, $xyz\ne vxy$, otherwise we have a contradiction to $a_3=2$ by a similar argument given in the last paragrpah.
By symmetry we may assume that $xyz=vyz$.
Then $vz=\alpha$ and $vy=\gamma$.
Let $\delta:=uv$. Then $A_3(X)=\{\alpha\beta\gamma, \gamma\gamma\delta\}$.
If $zu=\alpha$, then $uzv=\alpha\alpha\delta\notin \{\alpha\beta\gamma, \gamma\gamma\delta\}$, a contradiction.
Therefore, $uz=\gamma$ and $\beta=\delta$, and hence $vx=\beta$,
which implies  $xuv=\beta\beta\beta\notin \{\alpha\beta\gamma, \gamma\gamma\delta\}$, a contradiction.
\end{proof}

\begin{lem}\label{lem:basic2}
Suppose that $|A_2(Y)|\leq 2$ for each $Y\in \binom{X}{3}$, i.e.,
every element of $A_3(X)$ is an isosceles triangle.
For all $x,y,z\in X$ with $xy\ne yz=zx$ and each $u\in X\setminus\{x,y,z\}$ we have the following:
\begin{enumerate}
\item If $ux\notin \{xy,xz\}$, then $A_2(X)=\{ux,xy,yz\}$ and $a_3=3$;
\item If $uy\notin \{xy,yz\}$, then $A_2(X)=\{uy,xy,yz\}$ and $a_3=3$;
\item If $uz\notin \{xy,yz\}$, then $A_2(X)=\{uz,xy,yz\}$ and $a_3=3$.
\end{enumerate}
\end{lem}
\begin{proof}
For short we put
\[\mbox{$\alpha:=d(x,y)$, $\beta:=d(y,z)$ and $\gamma:=d(u,x)$.}\]

(i) Suppose $\gamma\notin \{\alpha,\beta\}$ and put $\mu:=uy$ and $\nu:=uz$.
Then $uxy=\alpha\gamma\mu$, $uxz=\beta\gamma\nu$ and $xyz=\alpha\beta\beta$ are distinct elements
of $A_3(X)$ since $\mu\in \{\alpha,\gamma\}$ and $\nu\in \{\beta,\gamma\}$ by the assumption.
Therefore, $A_2(X)=\{\alpha,\beta,\gamma\}$.

(ii) By symmetry the proof of (ii) is parallel to that of (i).

(iii) We claim that, for all distinct $x_1,x_2,x_3,x_4\in X$,
if $x_1x_2\ne x_1x_3=x_2x_3$, then we may assume that $x_4x_1, x_2x_4\in \{x_1x_2,x_2x_3\}$,
otherwise, $A_2(X)=\{x_1x_2,x_2x_3, x_4x_1\}$ by (i) and (ii), and so (iii) holds.

Suppose $\gamma:=uz\notin \{\alpha,\beta\}$ and $v\in X\setminus \{x,y,z,u\}$.
By the above claim, $ux,uy,vx,vy\in \{\alpha,\beta\}$. On the other hand, $ux,uy\in \{\beta,\gamma\}$ since both of $uxz$ and $uzy$
should be isosceles. Therefore, $ux=uy=\beta$.
By the above claim, $uv,zv\in \{\beta,\gamma\}$.
If $uzv\in \{\gamma\gamma\gamma,\gamma\gamma\beta\}$, then
$xyz=\beta\beta\alpha$, $xuz=\beta\beta\gamma$ and $uzx$ are distinct elements of $A_3(X)$,
and hence $A_2(X)=\{\alpha,\beta,\gamma\}$ as desired.
Thus, we may assume that $vu=vz=\beta$.
Recall that $xv,yv\in \{\alpha,\beta\}$.
If $xv=\beta$, then $A_3(X)=\{\beta\beta\alpha, \beta\beta\gamma, \beta\beta\beta\}$, and hence
$A_2(X)=\{\alpha,\beta,\gamma\}$ as desired.
Similarly, if $yv=\beta$, then we have $A_2(X)=\{\alpha,\beta,\gamma\}$.
Therefore, we may assume that $xv=yv=\alpha$.
Then we have the same conclusion since $xyz=\alpha\alpha\alpha$.
This completes the proof.
\end{proof}

\begin{thm}\label{thm:3}
If $a_3\leq 3$ and $5\leq n$, then $a_2\leq a_3$.
\end{thm}
\begin{proof}
Suppose the contrary, i.e., $a_2>a_3$.
By Lemma~\ref{lem:xyz}(iv), $A_3(X)$ consists of only isosceles triangles.
Since $a_2>1$ by Theorem~\ref{thm:a1}, we can take $xyz\in A_3(X)$ with $xy\ne yz=xz$.
Applying the contraposition of Lemma~\ref{lem:basic2} we obtain that the distances from $\{x,y,z\}$ to each point in $X\setminus\{x,y,z\}$
are contained in $\{d(x,y),d(x,z)\}$.

Let $u,v\in X\setminus\{x,y,z\}$ with $d(u,v)\notin \{d(x,y),d(x,z)\}$.
Since $uvx$ is an isosceles triangle, it follows from the result in the last paragraph that
$ux=vx\in \{xy,yz\}$, each of which induces $d(u,v)\in  \{d(x,y),d(x,z)\}$ by Lemma~\ref{lem:basic2}, a contradiction.
\end{proof}

\begin{thm}\label{thm:25}
If $a_3=2$, $n\ge 5$, then for some $\alpha\in A_2(X)$
$(X,E_\alpha)$ is isomorphic to one of the following:
\begin{enumerate}
\item a complete bipartite graph;
\item the complement of a matching on $X$;
\item the pentagon.
\end{enumerate}
\end{thm}
\begin{proof}
By Theorem~\ref{thm:3}, $a_2=2$, and hence we assume that $A_2(X)=\{\alpha,\beta\}$ where $\alpha\ne \beta$.
Suppose that the graph $(X,E_\beta)$ has exactly two components, one of which has more than two points.
Let $X_1$ and $X_2$ be the distinct connected components of $(X,E_\beta)$ with $|X_1|\geq 3$.
Then $x_1x_2=\alpha$ for all $x_1\in X_1$ and $x_2\in X_2$.
We claim that $xy=\beta$ for all $x,y\in X_i$ and $i=1,2$.
Otherwise, there exist $x,y,z\in X_i$ such that $xy=yz=\alpha$ and $xz=\beta$,
which implies $A_3(X)=\{\alpha\alpha\beta,\alpha\alpha\alpha,\beta\beta\alpha\}$, a contradiction to $a_3=2$.
Therefore, $(X,E_\alpha)$ is complete bipartite.

Suppose that the graph $(X,E_\beta)$ has more than two components.
Notice that one of the components contains more than one points and
hence $A_3(X)=\{\alpha\alpha\alpha,\alpha\alpha\beta\}$. This implies that
each component contains at most two points. Therefore, $(X,E_\beta)$ is a matching. 

Suppose that both $(X,E_\alpha)$ and $(X,E_\beta)$ are connected.
Since the Ramsey number $R(3,3)$ equals $6$ and $A_3(X)=\{\alpha\alpha\beta, \beta\beta\alpha\}$,
we have $|X|\leq 5$, and $(X,E_\alpha)$ is the pentagon by the assumption $|X|\geq 5$.
\end{proof}

If $\Gamma=(V,E)$ is a triangle-free finite graph and $d:V\times V\to \mathbb{R}_{\geq 0}$ is defined
by $d(x,y)=0$ if $x=y$, $d(x,y)=1$ if $\{x,y\}\in E$ and otherwise $d(x,y)=2$,
then $(V,d)$ is a finite metric space with $|A_2(V)|=2$ and $|A_3(V)|\leq 3$.
This implies that such triangle-free graphs except for the graphs or their complement given in Theorem~\ref{thm:25}
induce metric spaces with $a_2=2$ and $a_3=3$.

For the remainder of this section we assume that
$a_2=a_3=3$ and $n\geq 5$, and $A_2(X)=\{\alpha,\beta,\gamma\}$.

\begin{lem}\label{lem:trifree}
If all of $(X,E_\alpha)$, $(X,E_\beta)$ and $(X,E_\gamma)$ are triangle-free, then
$A_3(X)=\{\alpha\alpha\beta,\alpha\alpha\gamma,\beta\beta\gamma\}$ for a suitable ordering of $\alpha$, $\beta$ and $\gamma$.
\end{lem}
\begin{proof}
First, we assume that each of the three graphs has no vertex of degree at least three.
Then each of the three graphs is a disjoint union of paths and cycles with at most $n$ edges.
Therefore,
\[\binom{n}{2}=|\binom{X}{2}|=|E_\alpha|+|E_\beta|+|E_\gamma|\leq 3n,\]
and we obtain $n\leq 7$.
Since they are triangle-free, the equality holds if and only if
each of the three graphs is isomorphic to the cycle of length seven,
which implies $A_2(X)=M_2$, namely, every triangle is an isosceles.
But, such an edge-coloring of $K_n$ without any tricolored triangle does not exists
Thus, $n\leq 6$.
If $n=6$ and one of the three graphs is a cycle of length six,
then the three diagonal lines from a fixed vertex should be colored into the antipodal line and the others
because of $a_3=3$, but applying this argument for other vertices we have a mono-colored triangle, a contradiction.
Therefore, if $n=6$, then $|E_\alpha|=|E_\beta|=|E_\gamma|=5$ since any triangle-free graph with six edges and six vertices
is a cycle of length $6$. In this case we have $A_2(X)=M_2$, so that any tricolored triangle is forbidden.
This condition helps us to make a contradiction in constructing an edge-coloring.
Thus, $n\leq 5$, namely, $n=5$ by the assumption on $n$.
If $(X,E_\alpha)$ forms a pentagon, then it is impossible to decompose the complement of $E_\gamma$ into two parts
satisfying $a_3=3$. We have the similar contradiction for $(X,E_\beta)$ and $(X,E_\gamma)$.
Therefore, at least one of the three graphs has exactly four edges, and hence it is a path of length four.
But, it can be easily verified that such decompositions with $a_3=3$ do not occur.

Second, we assume that at least one of the three graphs has a vertex of degree at least three.
Without loss of generality we may assume that $R_\alpha(x)$ has distinct elements $u,v,w$.
Since $(X,E_\alpha)$ is triangle-free, $(u,v),(v,w),(w,u)\in R_\beta\cup R_\gamma$.
Since both of $(X,E_\alpha)$ and $(X,E_\beta)$ are triangle-free, it follows that
$\alpha\alpha\beta,\alpha\alpha\gamma\in A_3(X)$, and either $\beta\beta\gamma\in A_3(X)$ or $\gamma\gamma\beta\in A_3(X)$.
By symmetry between $\beta$ and $\gamma$ we obtain the formula given in this lemma.
\end{proof}

\begin{lem}\label{lem:tri}
If $\alpha\alpha\alpha \in A_3(X)$,
then $(X,E_\beta)$ or $(X,E_\gamma)$ is a matching on $X$.
\end{lem}
\begin{proof}
Suppose the contrary, i.e., $\beta,\gamma\in M_2$, so that $A_2(X)=M_2$ by $a_3=3$ and Lemma~\ref{lem:major}.
Then, by Lemma~\ref{lem:closed}, $(X,E_\alpha)$ is a disjoint union of cliques.
Since $\alpha\alpha\alpha\in A_3(X)$, there exists a clique in $(X,E_\alpha)$ with size at leats three.
Since $(X,E_\beta\cup E_\gamma)$ is complete multipartite,
it follows that $\beta\beta\alpha\in A_3(X)$ or $\gamma\gamma\alpha\in A_3(X)$.
Without loss of generality we may assume that $\beta\beta\alpha\in A_3(X)$.
Since $\gamma\in M_2$, it follows from Lemma~\ref{lem:closed} that
$(X,E_\alpha\cup E_\beta)$ is a disjoint union of cliques, which implies that
$\gamma\gamma\beta, \gamma\gamma\alpha\in A_3(X)$, which contradicts $a_3=3$.
\end{proof}

\begin{lem}\label{lem:tri2}
If $\alpha\alpha\alpha \in A_3(X)$ and $\beta,\gamma\notin M_2$,
then $A_3(X)=\{\alpha\alpha\alpha,\alpha\alpha\beta,\alpha\alpha\gamma\}$ for a suitable ordering of $\alpha$, $\beta$ and $\gamma$.
\end{lem}
\begin{proof}
Let $(x,y)\in R_\beta$. 
Since $n\geq 5$, it follows that there exists $z\in X$ such that $(x,z),(y,z)\in R_\alpha$.
This implies that $\alpha\alpha\beta\in A_3(X)$. 
Similarly, we have $\alpha\alpha\beta\in A_3(X)$ by symmetry between $\beta$ and $\gamma$. 
\end{proof}

\begin{lem}\label{lem:tri3}
If $\alpha\alpha\alpha \in A_3(X)$ and $\beta\in M_2$,
then $A_3(X)=\{\alpha\alpha\alpha,\beta\beta\alpha,\alpha\beta\gamma\}$ or
$A_3(X)=\{\alpha\alpha\alpha,\beta\beta\alpha,\beta\beta\gamma\}$
for a suitable ordering of $\alpha$, $\beta$ and $\gamma$.
\end{lem}
\begin{proof}
Since $\alpha,\beta\in M_2$, it follows from Lemma~\ref{lem:tri} that $\gamma \notin M_2$.
Thus, any element $[Y]\in A_3(X)$ with $\gamma\in A_2(Y)$ is one of the following:
\[ \alpha\beta\gamma,\alpha\alpha\gamma, \beta\beta\gamma.\]
If $\alpha\beta\gamma\in A_3(X)$, then $(X,E_\alpha)$ is a disjoint union of cliques by Lemma~\ref{lem:closed} and $\beta\in M_2$.
Since $\gamma\notin M_2$, it follows that
$\beta\beta\alpha\in A_3(X)$ as desired.

If $\alpha\alpha\gamma\in A_3(X)$,
then $(X,E_\alpha\cup E_\gamma)$ is a disjoint union of cliques by Lemma~\ref{lem:closed} and $\beta\in M_2$.
This implies that $\beta\beta\alpha, \beta\beta\gamma\in A_3(X)$, a contradiction to $a_3=3$.

Finally, we assume that
\[\mbox{$\beta\beta\gamma\in A_3(X)$ and $\alpha\beta\gamma,\alpha\alpha\gamma\notin A_3(X)$.}\]
Then the remaining element in $A_3(X)$ is one of the following:
\[\alpha\alpha\gamma,\beta\beta\beta, \beta\beta\alpha, \alpha\alpha\beta.\]
Applying Lemma~\ref{lem:closed} for the first three cases
we obtain that $(X,E_\alpha\cup E_\gamma)$ is a disjoint union of cliques for the first case,
which implies $\beta\beta\alpha\in A_3(X)$ as desired.
Also, for the last case we obtain from Lemma~\ref{lem:closed} that
$(X,E_\beta\cup E_\gamma)$ is a disjoint union of cliques, which implies
$\alpha\alpha\gamma\in A_3(X)$, a contradiction to $a_2=2$.
\end{proof}

\begin{cor}\label{cor:40}
The set $A_3(X)$ forms
one of the following for a suitable ordering of $\alpha,\beta$ and $\gamma\in A_2(X)$:
\begin{enumerate}
\item $\{\alpha\alpha\beta,\alpha\alpha\gamma, \beta\beta\gamma\}$;
\item  $\{\alpha\alpha\alpha, \alpha\beta\gamma, \beta\beta\alpha\}$;
\item $\{\alpha\alpha\alpha, \alpha\alpha\beta, \alpha\alpha\gamma\}$;
\item $\{\alpha\alpha\alpha, \beta\beta\alpha, \beta\beta\gamma\}$.
\end{enumerate}
\end{cor}
\begin{proof}
By Lemma~\ref{lem:trifree}, if each of $(X,E_\alpha)$, $(X,E_\beta)$ and $(X,E_\gamma)$ is triangle-free,
then (i) holds.

Suppose one of the three graphs contain a triangle.
We may assume that $(X,E_\alpha)$ contains a triangle, equivalently, $\alpha\alpha\alpha\in A_3(X)$.
By Lemma~\ref{lem:tri}, \ref{lem:tri2} and \ref{lem:tri3},
we have only (ii), (iii) and (iv).
\end{proof}

\begin{ex}\label{ex:1}
Let $\{Y,Z\}$ be a bipartition of $X$ with $|Y|=|Z|=4$, 
$E_\gamma$ a perfect matching on $X$ contained in $\binom{Y}{2}\cup\binom{Z}{2}$,
$E_\beta=(\binom{Y}{2}\cup \binom{Z}{2})\setminus E_\gamma$ and
$E_\alpha$ the complement of $E_\beta\cup E_\gamma$.
For each subset $W$ of $X$, if $|W|\geq 5$, then
$A_3(W)=\{\alpha\alpha\beta,\alpha\alpha\gamma,\beta\beta\gamma\}$.
\end{ex}

\begin{ex}\label{ex:2}
Let $\{Y,Z\}$ be a bipartition of $X$.
We define $E_\alpha=\binom{Y}{2}\cup \binom{Z}{2}$, $E_\gamma$ to be a matching between $Y$ and $Z$, and $E_\beta$ to be the complement
of $E_\alpha\cup E_\gamma$. 
Then $A_3(X)=\{\alpha\alpha\alpha,\alpha\beta\gamma,\beta\beta\alpha\}$.  
\end{ex}

\begin{ex}\label{ex:3}
Let $(X,E_\beta)$ and $(X,E_\gamma)$ be non-empty matchings on $X$ such that
$(X,E_\beta\cup E_\gamma)$ is also a matching on $X$, and $E_\alpha$
the complement of $E_\beta\cup E_\gamma$. Then $A_3(X)=\{\alpha\alpha\alpha,\alpha\alpha\beta,\alpha\alpha\gamma\}$.
\end{ex}

\begin{ex}\label{ex:4}
Let $\{Y, Z\}$ be a bipartition of $X$ with $|Z|=2$.
We define $E_\alpha=\binom{Y}{2}$, $E_\gamma=\binom{Z}{2}$ and $E_\beta$ to be the complement
of $E_\alpha\cup E_\gamma$. Then $A_3(X)=\{\alpha\alpha\alpha,\beta\beta\alpha,\beta\beta\gamma\}$.
\end{ex}
\begin{thm}\label{thm:a4}
If $a_2=a_3=3$ and $5\leq n$, then $(X,d)$ is isomorphic to
a metric space given in Example~\em{\ref{ex:1}-\ref{ex:4}}.
\end{thm}
\begin{proof}
By Corollary~\ref{cor:40}, $A_3(X)$ is determined to be the cases (i)-(iv) given in Corollary~\ref{cor:40}.

Suppose $A_3(X)=\{\alpha\alpha\beta,\alpha\alpha\gamma, \beta\beta\gamma\}$.
By Lemma~\ref{lem:closed}, $(X,E_\beta\cup E_\gamma)$ is a disjoint union of cliques.
Since $\alpha\alpha\alpha\notin A_3(X)$, $(X,E_\beta\cup E_\gamma)$ has exactly two cliques.
Since $(X,E_\gamma)$ is a matching on $X$, it follows from $\beta\beta\gamma\in A_3(X)$ that
each vertex in $(X,E_\beta)$ has degree at most two. Therefore, we conclude from $\beta\beta\beta \notin A_3(X)$ that each connected component of $(X,E_\beta)$
is either a cycle of length 4 or a path of length at most three. Thus, $(X,E_\alpha)$ is
an induced subgraph of $K_{4,4}$.
Since $(X,E_\beta)$ is the complement of $(X,E_\alpha\cup E_\gamma)$,
it follows that $X$ is an induced subspace of the example given in Example~\ref{ex:1}.

Suppose $A_3(X)=\{\alpha\alpha\alpha, \alpha\beta\gamma, \beta\beta\alpha\}$.
Note that $(X,E_\alpha)$ is a disjoint union of exactly two cliques,
$(X,E_\gamma)$ is a matching on $X$ and $(X,E_\beta)$ is the complement of $(X,E_\alpha\cup E_\gamma)$,
that is exactly what is given in Example~\ref{ex:2}.

Suppose $A_3(X)=\{\alpha\alpha\alpha, \alpha\alpha\beta, \alpha\alpha\gamma\}$.
Since
$(X,E_\beta)$ and $(X,E_\gamma)$ is matchings on $X$ such that $Y\cap Z=\emptyset$ for all $Y\in E_\beta$ and $Z\in E_\gamma$,
the configuration $(X,\{E_\alpha,E_\beta,E_\gamma\})$ is the same as in Example~\ref{ex:3}.

Suppose $A_3(X)=\{\alpha\alpha\alpha, \beta\beta\alpha, \beta\beta\gamma\}$.
Note that $(X,E_\alpha)$ is a disjoint union of more than one cliques
such that any connected components of size more than two are connected by edges in $E_\beta$
and $(X,E_\gamma)$ is just one edge which connects two isolated points in $(X,E_\alpha)$.
This configuration is the same as in Example~\ref{ex:4}.
\end{proof}

\section{Embedding to Euclidean spaces}
For a finite metric space $(X,d)$ we say that $\rho:X\to \mathbb{R}^m$ is
a \textit{Euclidean embedding} of $X$ into $\mathbb{R}^m$ 
if $(X, d)$ is isomorphic to $(\rho(X),D)$ where $D$ is the Euclidean distance. 
We denote by $m_X$ the smallest non-negative integer $m$ such that there exists
a Euclidean embedding of $X$ into $\mathbb{R}^m$. 
Euclidean embeddings of $X$ with $a_2=2$ are studied in \cite{ES66, M, R, NS2}. 
In this section, we consider some Euclidean embeddings not only for the case where $a_2=2$. 

\begin{lem}[{\cite[Lemma~4.5]{NS1}}]\label{lem:NS}
Let $(X,d)$ be a finite metric space and $A,B\subseteq X$ such that
$d(a,b)$ is constant whenever $(a,b)\in A\times B$.
Then $m_{A\cup B}\geq m_A+m_B$.
\end{lem}
\begin{proof}
Let $\rho$ denote a Euclidean embedding of $A\cup B$ into $\mathbb{R}^m$ 
where $m:=m_{A\cup B}$.
Let $W$ be the subspace of $\mathbb{R}^m$ spanned by $\{\rho(a)-\rho(a_0)\mid a\in A\}$
where $a_0\in A$ is fixed.
We may assume $\rho(A)\subseteq W$ since we can translate $\rho$ to a similar Euclidean embedding if necessary.
Let $\pi$ denote the orthogonal projection from $\mathbb{R}^m$ to $W$.
Then $\rho(b)-\pi(\rho(b))$ is orthogonal to $W$ for each $b\in B$, and hence,
\[D(\pi(\rho(b)),\rho(a))^2=D(\rho(b),\rho(a))^2-D(\rho(b),\pi(\rho(b)))^2\]
is constant whenever $a\in A$ by the assumption. 
This implies that $\rho(A)$ lies in a sphere at the center $\pi(\rho(b))$.

We claim that $\pi(\rho(b_1))=\pi(\rho(b_2))$ for all $b_1,b_2\in B$.
Since $\rho(a)$, $\rho(a_1)$ and $\pi(\rho(b_i))$ form an isosceles for $a,a_1\in A$,
$\rho(a)-\rho(a_1)$ is orthogonal to $\pi(\rho(b_i))-\frac{\rho(a)+\rho(a_1)}{2}$ for
$i=1,2$. Thus, $\pi(\rho(b_1))-\pi(\rho(b_2))$ is orthogonal to $\rho(a)-\rho(a_1)$.
Since $a, a_1\in A$ are arbitrarily taken, it follows that
\[\pi(\rho(b_1))-\pi(\rho(b_2))\in W\cap W^\perp=\{0\}.\]

Again, we may assume that $\pi(\rho(b))$ is the zero vector in $\mathbb{R}^m$ since
we can translate $\rho$ to a similar Euclidean embedding with $\pi(\rho(b))=0$ if necessary.
Then $\rho(b)=\rho(b)-\pi(\rho(b))$ is orthogonal to $W$, implying 
$\rho(B)\subseteq W^\perp$.
Therefore, 
\[m_{A\cap B}=m= \dim(W)+\dim({W^\perp})\geq m_A+m_B.\]
\end{proof}

Now we exhibit Euclidean embeddings of several finite metric spaces appeared in our main results 
as follows:
\begin{enumerate}
\item $a_1=1$;

\item $a_2=2$ and one of the two elements of $A_2(X)$ induces the following graph;
\begin{enumerate}
\item The complete bipartite graph $K_{m,n-m}$; 
\item The matching consisting of only one edge;
\item The matching consisting of $m$ edges with $m>1$;
\item The pentagon.
\end{enumerate}
\item $a_3=3$;
\begin{enumerate}
\item $(X,d)$ given in Example~\ref{ex:1};
\item $(X,d)$ given in Example~\ref{ex:2} where $|Y|=|Z|=|E_\gamma|=m$ and $n=2m$;
\item $(X,d)$ given in Example~\ref{ex:3} where $|E_\beta|=p$, $|E_\gamma|=q$ and $n=2p+2q$ and $\max\{2,q\}\leq p$;
\item $(X,d)$ given in Example~\ref{ex:4},
\end{enumerate}
\end{enumerate}
We use the row vectors of a matrix to represent Euclidean embeddings by
corresponding the elements of $X$ to the row vectors.
For matrices $P$ and $Q$ over $\mathbb{R}$ we shall write the matrix
\[\begin{pmatrix}P &O\\O&Q\end{pmatrix}\] 
as $P\oplus Q$, and if $P$ and $Q$ have the same number columns,
\[\begin{pmatrix}P \\
Q\end{pmatrix}\]
as $(P;Q)$.

(i)  The row vectors of the matrix $I_n-\frac{1}{n}J_n$ where $I$ and $J$ are the identity matrix and the all-one matrix of degree $n$ give a Euclidean embedding into $\mathbb{R}^{n-1}$ isomorphic to 
$\{(x_1,x_2,\ldots, x_n)\in \mathbb{R}^n\mid \sum_{i=1}^nx_i=0\}$.

(ii)-(a) The row vectors of $(I_m-\frac{1}{m}J_m)\oplus (I_{n-m}-\frac{1}{n-m}J_{n-m})$
induces a Euclidean embedding into $\mathbb{R}^{n-2}$.

(ii)-(b) The matrix $(h;I_{n-1}-\frac{1}{n-1}J_{n-1})$ where $h$ is the negative of the first row of
$I_{n-1}-\frac{1}{n-1}J_{n-1}$ induces a Euclidean embedding into $\mathbb{R}^{n-2}$.

(ii)-(c) The matrix $(I_{n-m};(-I_m,O))$ where $O$ is the $m\times (n-2m)$ zero matrix induces 
a Euclidean embedding into $\mathbb{R}^{n-m}$.

We shall denote the $m\times 2$ matrix whose $i$th row equals
$(\cos(\frac{2\pi i}{m}), \sin(\frac{2\pi i}{m}))$ by $C_m$.

(ii)-(d) The matrix $C_5$ induces a Euclidean embedding into
$\mathbb{R}^2$.

(iii)-(a) The matrix $C_4\oplus C_4$  induces a Euclidean embedding into
$\mathbb{R}^4$.

(iii)-(b) The matrix $(I_m-\frac{1}{m}J_m; -I_m+\frac{1}{m}J_m)$ 
 induces a Euclidean embedding into
$\mathbb{R}^{m-1}$.
We remark that there exists $X\subset \mathbb{R}^m$ with $a_3(X)=3$ and $|X|=2m+2$ for any $m\ge 2$. For example, we have a regular pentagon and a cube for $m=2,3$. 

Let $c$ be a positive real number at most $\sqrt{2}$. 
Let 
\[CP_{2p,c}=\frac{1}{2p}
\begin{pmatrix} p\sqrt{2-c^2}I_{p}-2\sqrt{2-c^2}J_p & cpI_p\\
p\sqrt{2-c^2}I_{p}-2\sqrt{2-c^2}J_p &-cpI_p
\end{pmatrix}\]
and 
\[CP_{2q,c}^t=CP_{2q,c}+
\frac{t}{q}\begin{pmatrix}
J_q &O_q\\
J_q &O_q
\end{pmatrix}. 
\]

(iii)-(c) $CP_{2p,\sqrt{2}}\oplus CP_{2q,c}^t$ where $t=\sqrt{\frac{2-c}{4q}}$ induces
a Euclidean embedding into $\mathbb{R}^{n-p}$. 

(iii)-(d)
The matrix $(I_{n-2}-\frac{1}{n-2}J_{n-2})\oplus C_2$ induces a Euclidean embedding into $\mathbb{R}^{n-2}$.

Furthermore, all of the above embedding are into $\mathbb{R}^{m_X}$ since
Lemma~\ref{lem:NS} gives a lower bound for $m_X$ which may attain $m_Y+m_Z$ 
for a suitable choice of $Y$ and $Z$ in the following:

(i) The characteristic polynomial of $G$ given in (\ref{eq:G}) equals
\[t(t-\alpha^2)^{n-1}\]
where $A_2(X)=\{\alpha\}$.
Following the criterion in \cite{Ne} $m_X=n-1$.

(ii)-(a) Let $Y$ and $Z$ be the bipartition of $K_{m,n-m}$ with $|Y|=m$ and $|Z|=n-m$.
Since $|A_2(Y)|=|A_2(Z)|\leq 1$, we have $m_Y\geq m-1$ and $m_Z\geq n-m-1$.
By Lemma~\ref{lem:NS}, $m_X\geq n-2$, which attains the dimension given above.

(ii)-(b),(c) Since the complement of a matching is complete multipartite,
it follows from \cite{ES66} that there exists a unique Euclidean embedding into $\mathbb{R}^{n-1}$ up to similarity. Therefore, $m_X=n-2$ in the case (b),
and $m_X=n-m$ in the case of (c). 

(ii)-(d) We have $m_X=2$ since it is clear that there no Euclidean embedding into $\mathbb{R}$.

(iii)-(a) By Lemma~\ref{lem:NS}, $m_X\geq m_Y+m_Z=4$, which attains the dimension 
given above.

(iii)-(b) By Lemma~\ref{lem:NS}, $m_X\geq m_Y=m-1$, which attains the dimension 
given above.

(iii)-(c) Note that $d(y,z)=\alpha$ for any $(y,z)\in Y\times Z$ where 
$Y=\{x\in X\mid v(x,X)_\beta=1\}$ and $Z=\{x\in X\mid v(x,X)_\gamma=1\}$. 
Let $\rho$ denote a Euclidean embedding of $X=Y\cup Z$. 
If $\rho$ satisfy $A_2(\rho (Y))=\{c,1\}$ and $D(\rho (a),\rho (b))=1$ if $d(a,b)=\alpha$, 
then $\rho(Y)$ is isometric to $CP_{2p,c}$ for some $c$. 
Then we notice that $\rho (X)$ is isometric to 
\[ CP_{2p,c}\oplus CP_{2q,c'}^t\]
for some suitable $c'$ and $t$. 
First, we assume that $c=\sqrt{2}$. 
Then $c'\ne \sqrt{2}$ since $c=c'$ implies $\beta=\gamma$. 
Since $D(\rho (y),\rho (z))=1$ for any $y\in Y$, $z\in Z$, we have $t=\sqrt{\frac{2-c'}{4q}}$. 
Then $\dim (\rho (X))\ge p+q=n-p$ where $\dim (X)$ is the dimension of the subspace 
spanned by $\{\rho (x)-\rho (x_0)\mid x\in X\}$ where $x_0\in X$ is fixed. 
Next, we assume that $c'=\sqrt{2}$. Then, by a similar argument,  
we have $\dim (X)\ge n-q\ge n-p$. 
Finally, we assume that $c\ne \sqrt{2}$ and $c\ne \sqrt{2}$. 
Then $\dim (\rho (X))\ge n-2\ge n-p$. 
Therefore $m_X\ge n-p$ which attains the dimension given above. 

(iii)-(d) By Lemma~\ref{lem:NS}, $m_X\geq m_Y+m_Z=n-3+1=n-2$, which attains the dimension 
given above.

Epstein, Lott, Miller and Palsson \cite{ELMP} discussed about a set with few distinct triangles.  In this paper, we adopt $a_3$ in stead of the number of distinct triangles, i.e., 
a set of collinear three points is also regarded as a triangle in this paper. 
We define 
\[F_m (k,t) =\max \{|X| \mid X\subset \mathbb{R}^m, a_k(X)=t\}.\]

By the above argument of this section, we have the following theorem. 

\begin{thm} We have the following: 
\begin{enumerate}
\item $F_2(3,2)=5$ and $F_m(3,2)=2m$ $(m\ge 3)$ hold, 
and optimal sets are the pentagon and cross polytopes, respectively; 
\item $2m+2\le F_m(3, 3) \le \max \{2m+2, F_m(2,2)\}$.
\end{enumerate}
\end{thm}

Lison\v{e}k\cite{L97} determined $F_m(2,2)$ for $m\le 8$. 

\begin{thm}
$F_2(2,2)=5, F_3(2,2)=6, F_4(2,2)=10, F_5(2,2)=16, F_6(2,2)=27, F_7(2,2)=29$ and $F_8(2,2)=45.$  

In particular, the unique optimal set $X$ in $\mathbb{R}^4$ and $\mathbb{R}^5$ satisfies $a_3(X)=4$ and $a_3(X)=3$, respectively. 
\end{thm}

By combining the above theorems, we have the following. 

\begin{cor}
\[\mbox{$F_2(3,3)=6, F_3(3,3)=8, F_4(3,3)=10$ and $F_5(3,3)=16.$}\]
In particular, every optimal set in $\mathbb{R}^m$ for $m=2,3,4$ is given by (iii)-(b) of this section. 
\end{cor}

\section*{Acknowledgement}
This research was supported by Basic Science Research Program through the National Research Foundation of Korea(NRF) funded by the Ministry of  Science, ICT \& Future Planning(NRF-2016R1A2B4013474). 
The second author was supported by JSPS KAKENHI Grant-in-Aid for Scientific Research (C) 26400003. We also thank Sho Suda for helpful comments. 


\end{document}